\documentclass[a4paper,leqno,12pt]{amsart}

\usepackage{epsf,graphicx,epsfig}
\usepackage{amscd}
\usepackage{amsmath,latexsym,amssymb,amsthm}
\usepackage[nospace,noadjust]{cite}
\usepackage{setspace,cite}
\usepackage{lscape,fancyhdr,fancybox}
\usepackage{graphicx}

\usepackage{color}
\usepackage{url}
\usepackage{enumerate}
\usepackage[mathscr]{euscript}

  \theoremstyle{plain}

\swapnumbers
    \newtheorem{thm}{Theorem}[section]
    
   \newtheorem{lemma}[thm]{Lemma}
    \newtheorem{corollary}[thm]{Corollary}
    
    \newtheorem{subsec}[thm]{}
\theoremstyle{definition}
    \newtheorem{defn}[thm]{Definition}
\theoremstyle{remark}
     \newtheorem{remark}[thm]{Remark}

\title{}
\author{}
\date{}
\usepackage{amssymb}

\begin{document}
\title{Some special property of Farey sequence}

\author{Ripan Saha}
\email{ripanjumaths@gmail.com}
\address{Department of Mathematics,
Raiganj University, Raiganj, 733134,
West Bengal, India.}

\date{\today}
\subjclass[2010]{$11$Axx, $11$Bxx.}

\thispagestyle{empty}

\begin{abstract}
We discuss some special property of the Farey sequence. We show in each term of the Farey sequence, ratio of the sum of elements in the denominator and the sum of elements in the numerator is exactly two. We also show that the Farey sequence contains a palindromic structure.
\end{abstract}
\maketitle
\section{Introduction} 
Farey sequence is named after British geologist John Farey, Sr., who publised a result in  Philosophical Magazine in 1816 about these sequence without giving a proof. Later, Cauchy proved the result conjectured by Farey. Though,  Charles Haros proved similar result in 1802 which were not known to Farey and Cauchy. Later, Farey sequence appeared in many different areas of mathematics, including number theory, topology, geometry, see, \cite{4},\cite{2},\cite{0}. Farey sequence is also related to Ford’s circle and Riemann hypothesis,\cite{5}.

Farey sequence \cite{4}, $F_n$ in the interval $[0,1]$ is defined as sequence of ascending rational numbers in reduced form starting at $\frac{0}{1}$ and ending at $\frac{1}{1}$ such that the elements in the $n$th term, denominator
are less than or equal to $n$. First few terms of Farey sequence are listed as follows,
\begin{align*}
&F_1=\frac{0}{1}, \frac{1}{1},\\
&F_2=\frac{0}{1}, \frac{1}{2}, \frac{1}{1},\\
&F_3=\frac{0}{1}, \frac{1}{3}, \frac{1}{2}, \frac{2}{3}, \frac{1}{1},\\
&F_4=\frac{0}{1}, \frac{1}{4}, \frac{1}{3}, \frac{1}{2}, \frac{2}{3}, \frac{3}{4}, \frac{1}{1},\\
&F_5=\frac{0}{1}, \frac{1}{5}, \frac{1}{4}, \frac{1}{3}, \frac{2}{5}, \frac{1}{2}, \frac{3}{5}, \frac{2}{3}, \frac{3}{4}, \frac{4}{5}, \frac{1}{1},\\
&\cdots
\end{align*}
Suppose $\frac{a}{b}$ and $\frac{c}{d}$ are two fractions in Farey sequence. The mediant or sum of $\frac{a}{b}$ and $\frac{c}{d}$ is defined as
$$\frac{a}{b}\oplus\frac{c}{d}=\frac{a+c}{b+d}$$
If $\frac{a}{b}<\frac{c}{d}$ then $\frac{a}{b}<\frac{a+c}{b+d}<\frac{c}{d}$, see\cite{0}.

For all the terms of the Farey sequence given above, one can note that sum of the elements in the denominator is always two times of the sum of the elements in the numerator. In $F_1$, sum of numerator is $1$ and sum of denominator is $2$. In $F_2$, sum of numerator is $2$ and sum of denominator is $4$. Likewise in $F_3,F_4,F_5$, sum of denominator is always two times the sum of numerator. So, one may ask if this result holds for all $F_n$ or not. In this paper, we will show that in general this result is true. Surprisingly, no one has noticed this special property of Farey sequence. In the second section of this article, we introduce this special property of Farey sequence. In the final section, we show that  denominators of each fraction in  $F_n$ has a palindromic structure for all $n$.
\section{A special property of Farey sequence}
In this section, we prove our first result on Farey sequence.
\begin{lemma}\label{lemma1}
\cite{0} The length of each term of Farey sequence $F_n$ is given by the following recurring formula:
$$|F_n|=|F_{n-1}|+\phi(n).$$
Where $\phi$ is the Euler\rq s Totient function, which is number of elements less than and co-prime to $n$.
\end{lemma}
\begin{remark}\label{rmk1}
From the Lemma \ref{lemma1}, we get the information that number of new elements appears in $F_n$ compare to $F_{n-1}$ is $\phi(n)$. The numbers $1\leq k<n$ which are co-prime to $n$, appears to the numerator in the new terms of $F_n$ and the denominator of each new terms in $F_n$ is always $n$.
\end{remark}
\begin{lemma}\cite{1}\label{lemma2}
For any positive intergers $n$ and $k$,
$$\sum_{\text{gcd}(n,k)=1}k=\frac{n\phi(n)}{2}.$$
\end{lemma}
\begin{proof}
If $k$ is co-prime to $n$, then $n-k$ is also co-prime to $n$. Note that, $k$ can not be equal to $n-k$, otherwise, $\text{gcd}(k,n)$ will not be $1$. The number of elements co-prime to $n$ is $\phi(n)$. So pairing $k$ and $n-k$, we get the total sum is $\frac{n\phi(n)}
{2}$ .
\end{proof}
\begin{thm}
In each term of the Farey sequence, the sum of the elements in the denominator is always two times of the sum of  elements in the numerator.
\end{thm}
\begin{proof}
We prove this result by induction on $n$. Let $N_n$ denotes the sum of the elements in the numerator of the $n$th term of Farey sequence, that is, in $F_n$ and $D_n$ denotes the sum of  elements in the
denominator of $n$th term of Farey sequence.
For $n = 1$,
$$N_1=1\,\,\, \text{and}\,\,\, D_1=2$$
So, our claim is true for $n = 1$.
Now suppose result is true for $n-1$.
We show that result is true for all $n$,
we have,
\begin{align*}
&N_n=N_{n-1}+\sum_{\text{gcd}(n,k)=1}k,\,\,\,\text{using remark}\,\ref{rmk1}.\\
&\hspace{1.5cm}=N_{n-1}+\frac{n\phi(n)}{2},\,\,\,\text{using lemma}\,\ref{lemma2}.\\
&D_n=D_{n-1}+n\phi(n),\,\,\,\text{using remark}\,\ref{rmk1}.\\
&\hspace{1.5cm}=2N_{n-1}+n\phi(n),\,\,\,\text{using induction}.\\
&\hspace{1.5cm}=2\Big(N_{n-1}+\frac{n\phi(n)}{2}\Big).\\
&\hspace{1.5cm}=2N_n.
\end{align*}
So, by induction our result follows.
\end{proof}
\section{Palindromic structure in Farey sequence}
In this final section, we show that there is a palindromic structure in denominators of $F_n$ for all $n$.
\begin{defn}
A word or sequence is called a Palindrom if it reads same from forward and backward direction.
\end{defn}
For example, the english word BOB, MADAM and numbers 1001,12321 are examples of palindroms.
\begin{defn}
Two fractions $\frac{a}{b}$ and $\frac{c}{d}$ in $F_n$ are called Farey neighbours if $\frac{a}{b}<\frac{c}{d}$ and there is no fraction in $F_n$ between them.
\end{defn}
The following lemma gives a condition for two fraction to be Farey neighbour.
\begin{lemma}\label{lemma3}
\cite{0} Let $\frac{0}{1}\leq \frac{a}{b}<\frac{c}{d}\leq \frac{1}{1}$ then $\frac{a}{b},\frac{c}{d}\in F_n$ are neighbours in the Farey sequence if and only if $bc-ad=1.$
\end{lemma}
\begin{corollary}\label{cor1}
If $\frac{0}{1}\leq\cdots<\frac{a_1}{b_1}<\frac{a_2}{b_2}<\cdots\leq\frac{1}{2}$ are Farey neighbours, then $\frac{1}{2}\leq \cdots
<\frac{b_2-a_2}{b_2}<\frac{b_1-a_1}{b_1}<\cdots\leq\frac{1}{1}$ are also Farey neighbours.
\end{corollary}
\begin{proof}
Using lemma \ref{lemma3}, on the given first Farey neighbours, we have,
\begin{align}
b_1a_2-a_1b_2=1.
\end{align}
So, $b_2(b_1-a_1)-b_1(b_2-a_2)=(b_1a_2-a_1b_2)=1.$

Now, $\frac{b_1-a_1}{b_1}=1-\frac{a_1}{b_1}\geq \frac{1}{2}$ as $\frac{0}{1}\leq\frac{a_1}{b_1}\leq \frac{1}{2}$. Similarly, $\frac{1}{2}\leq\frac{b_2-a_2}{b_2}\leq \frac{1}{1}$. Thus, our corollary follows.
\end{proof}
\begin{thm}
Denominators of each fraction in $F_n$ for all $n$ in a Farey sequence is a palindrom.
\end{thm}
\begin{proof}
We prove this theorem by induction.

In $F_1$, denominators are 1,1, which is a palindromic sequence.

In $F_2$, denominators are 1,2,1, which is also a palindromic sequence.

Now, suppose that denominators in $F_{n-1}$ are in a palindromic sequence. We need to show that denominators in $F_n$ is also in a palindromic sequence.
Using corollary \ref{cor1} and palindromic structure of denominators in $F_{n-1}$, we can write $F_{n-1}$ in the following form,
$$F_{n-1}=\lbrace\frac{0}{1},\cdots,\frac{a_1}{b_1},\frac{a_2}{b_2},\cdots,\frac{1}{2},\cdots,\frac{b_2-a_2}{b_2},\frac{b_1-a_1}{b_1},\cdots,\frac{1}{1}\rbrace.$$
Suppose in the next term $F_n$, a new term appears between $\frac{a_1}{b_1}$ and $\frac{a_2}{b_2}$. So, we can write $F_n$ from $F_{n-1}$ as,
\begin{align*}
\lbrace \frac{0}{1},\cdots,\frac{a_1}{b_1},\frac{a_1+a_2}{b_1+b_2},\frac{a_2}{b_2},\cdots,\frac{1}{2},\cdots,\frac{b_2-a_2}{b_2},\frac{(b_2-a_2)+(b_1-a_1)}{b_2+b_1},\frac{b_1-a_1}{b_1},\cdots,\frac{1}{1}\rbrace
\end{align*}
Let $a_1+a_2=k$. We have following relations by using remark \ref{rmk1},
$$b_1+b_2=n\,\,\,\text{and}\,\,\,\text{gcd}(n,k)=1.$$
Now, $(b_1-a_1)+(b_2-a_2)=(b_1+b_2)-(a_1+a_2)=n-k$.  Thus, $F_n$ can be written in terms of $n,k$ as follows:
$$\lbrace \frac{0}{1},\cdots,\frac{a_1}{b_1},\frac{k}{n},\frac{a_2}{b_2},\cdots,\frac{1}{2},\cdots,\frac{b_2-a_2}{b_2},\frac{n-k}{n},\frac{b_1-a_1}{b_1},\cdots,\frac{1}{1}\rbrace.$$
We know that if $\text{gcd}(n,k)=1$ then $\text{gcd}(n,n-k)=1$.  
Now, we need to show $\frac{b_2-a_2}{b_2}<\frac{(b_2-a_2)+(b_1-a_1)}{b_1+b_2}<\frac{b_1-a_1}{b_1}$ are Farey neighbours. So, we need to only check the condition of lemma \ref{lemma3}.  As $\frac{a_1}{b_1}<\frac{k}{n}<\frac{a_2}{b_2}$ are Farey neighbours, we have,
\begin{align}
\label{FNC}kb_1-na_1=1\,\,\,\text{and}\,\,\,na_2-kb_2=1.
\end{align}
Note that,
\begin{align*}
&n(b_1-a_1)-(n-k)b_1=kb_1-na_1=1\\
&(n-k)b_2-n(b_2-a_2)=na_2-kb_2=1.
\end{align*}
So, $\frac{b_2-a_2}{b_2}<\frac{n-k}{n}<\frac{b_1-a_1}{b_1}$ are Farey neighbours. Thus, we have proved that denominators in $F_n$ are palindrom.
Our desired result follows from the induction.
\end{proof}
\begin{remark}
By length of a palindrom, we mean the number of element apearing in the palindrom.
From the lemma \ref{lemma1}, it is clear that palindromic length of denominator palindrom in $F_n$ is $|F_{n-1}|+\phi(n)$. Thus, using Farey sequence one can get very large palindrom numbers.
\end{remark}
\section*{Acknowledgements} 
The author was supported by DST-INSPIRE, Govt. of India, fellowship during this work and would like to thank Sumit Mishra of Emory University for introducing Farey sequence to him.
 
\makeatletter
\renewcommand{\@biblabel}[1]{[#1]\hfill}
\makeatother


\begin{thebibliography}{99}


\bibitem{0}Ainsworth,J.; Dawson, M.; Pianta, J. and Warwick, J. “The farey sequence,” 2012.
\bibitem{1} Burton, David M. Elementary Number Theory, \textit{Allyn and Bacon, Inc.}, 1980, pp.150.
\bibitem{2}Cobeli, C. and Zaharescu, A. “The haros-farey sequence at two hundred years,” \textit{Acta Univ. Apulensis Math. Inform}, vol. 5, pp. 1– 38, 2003.
\bibitem{3}Hatcher, A. Topology of numbers, Chapter 1, \textit{http://www.math.cornell.edu/~hatcher}.
\bibitem{4} Hardy, G. H. and Wright, E. M., An Introduction to the Theory of Numbers, \textit{The Clarendon Press, Oxford University Press}, New York, Fifth edition (1979).
\bibitem{5}
Kanemitsu S. and Yoshimoto, M. Farey series and the Riemann hypothesis, \textit{Acta Arith.}, 75.
(1996), no. 4, 351-374.
\end{thebibliography}
\end{document}